\documentclass[a4paper,11pt]{amsart}
\usepackage{graphicx}
\usepackage[ansinew]{inputenc}    
\usepackage{fancyhdr}

\usepackage{color}

\usepackage{amsmath,amstext,amssymb,}
\usepackage{mathrsfs,amscd}

\newtheorem{thm}{Theorem}[section]
\newtheorem{cor}[thm]{Corollary}
\newtheorem{lem}[thm]{Lemma}
\newtheorem{prop}[thm]{Proposition}

\theoremstyle{definition}
\newtheorem{defn}[thm]{Definition}

\theoremstyle{definition}
\newtheorem{rem}[thm]{Remark}

\theoremstyle{definition}
\newtheorem{exam}[thm]{Example}

\def\R{\mathbb R}

\title[Affine metrics of locally strictly convex surfaces]{Affine metrics of locally strictly convex surfaces in affine 4-space}

\author{Juan J. Nuño-Ballesteros, Luis Sánchez}

\address{Departament de Geometria i Topologia,
Universitat de Val\`encia, Campus de Burjassot, 46100 Burjassot
SPAIN}

\email{Juan.Nuno@uv.es}

\address{Universidade de S\~ao Paulo, ICMC-SME, Caixa Postal 668,
13560-970 S\~ao Carlos (SP), BRAZIL}

\email{luisesan@icmc.usp.br}

\thanks{The first author has been partially supported by DGICYT Grant MTM2012-33073. The second author has been partially supported by FAPESP Grant BEPE 2011/21126-7.}

\keywords{Affine metric, affine normal plane, affine semiumbilical}

\subjclass[2000]{Primary 53A15; Secondary 53A07, 58K05}

\begin{document}

\maketitle

\begin{abstract}
We introduce a new family of affine metrics on a locally strictly convex surface $M$ in affine 4-space. Then, we define the symmetric and antisymmetric equiaffine planes associated with each metric. We show that if $M$ is immersed in a locally strictly convex hyperquadric, then the symmetric and the antisymmetric planes coincide and contain the affine normal to the hyperquadric. In particular, any surface immersed in a locally strictly convex hyperquadric is affine semiumbilical with respect to the symmetric or antisymmetric equiaffine planes.
\end{abstract}

\section{Introduction}

The main purpose in affine differential geometry is the study of properties of submanifolds $M^n$ in $m$-affine space that are invariant under the group of unimodular affine transformations. The main results for the classical theory of affine hypersurfaces can be seen, for instance, in \cite{Da,Li,N}.

Concerning submanifolds of codimension 2 there are few results. Nomizu and Vrancken  in \cite{N-V} developed an affine theory for surfaces in $\R^4$. They used the affine metric of Burstin and Mayer \cite{BM}, which is affine invariant, to construct the affine normal plane. However, this affine metric and the corresponding affine normal plane present several problems if the surface is locally strictly convex (i.e., at each point $p\in M$ there is a tangent hyperplane with a non-degenerate contact which locally supports $M$).
%
%

The first point is that in order to define the affine metric, we need that the surface is non-degenerate, in particular, $M$ cannot have inflections (points where the two second fundamental forms are collinear). But it is well known that any locally strictly convex compact surface $M$ with Euler characteristic $\chi(M)\ne0$ has at least an inflection, because of the Poincaré-Hopf formula (see \cite{MRR}).

Another point is that even if $M$ is non-degenerate, the affine metric is indefinite when $M$ is locally strictly convex. This is the opposite of what you expect, for instance, if $M$ is contained in a locally strictly convex hypersurface $N$, then the affine metric of $N$ is positive definite.

Finally, if $M$ is contained in a hypersurface $N$, you expect also some type of compatibility between the affine normals. This is important, for instance, if you want to consider contacts of the surface with affine hyperspheres. However, the affine normal plane to $M$ does not contain the affine normal vector to $N$ in general (see remark \ref{rem:exam}).

Our interest to study affine differential geometry of surfaces $M\subset\R^4$ comes from the understanding of the asymptotic configuration of $M$ near an inflection. There is a conjecture (see \cite{GR}) that any locally strictly convex surface homeomorphic to the sphere has at least two inflections. It is well known that a positive answer to this conjecture should imply a proof of the celebrated Carathéodory conjecture that any  convex compact surface $M\subset\R^3$ has at least two umbilic points (see for instance \cite{Bl}).

Partial proofs for the conjecture in $\R^4$ can be found in \cite{GMRR} for generic surfaces or in \cite{NB} for semiumbilical surfaces in the Euclidean sense (i.e., there is a non zero normal vector field whose shape operator is a multiple of the identity). The problem with the result in \cite{NB} is that the semiumbilical condition is not affine invariant, although the conjecture in $\R^4$ itself is affine in nature. Surfaces immersed in an Euclidean hypersphere or surfaces given by a product of two plane curves are examples of semiumbilical surfaces. However, surfaces immersed in other strictly convex hyperquadrics (like elliptic paraboloids or hyperboloids of two sheets) are not semiumbilical in general.

We introduce a new family of affine metrics $g_\xi$ on a locally strictly convex surface $M\subset\R^4$ which are positive definite. Here, $\xi$ is a transversal vector field such that $\xi$ and $T_pM$ span a local support hyperplane with non-degenerate contact at $p$. We show that when $M$ is immersed in a locally strictly convex hypersurface $N$, then there is a natural choice of $\xi$ in such a way that $g_\xi$ coincides with the Blasche metric of $N$ restricted to $M$.

For each affine metric $g_\xi$, we define the symmetric and the antisymmetric equiaffine planes, by using analogous arguments to that of Nomizu and Vrancken in \cite{N-V}. We also obtain algorithms to compute these normal planes.
The main result is that if $M$ is immersed in a locally strictly convex hyperquadric $N$, then the symmetric and the antisymmetric equiaffine planes coincide and contain the affine normal vector to $N$. As a consequence, any surface contained in a locally strictly convex hyperquadric is affine semiumbilical with respect to the symmetric or antisymmetric equiaffine planes. Another class of surfaces with the same property are those given by a product of two plane curves, hence our definition of affine semiumbilical surface has analogous properties as in the Euclidean case.

\section{Preliminaries}

Let $\R^4$ be the affine 4-space 
and $D$ the usual flat connection on $\R^4.$ Let $M\subset \R^4$ be an immersed surface and let
$\sigma$ be a transversal plane bundle on $M$. Then, for all
$p \in  M$,  $\sigma_{p}\subset T_p\R^4$   is a plane such that
\begin{equation*}
    T_p \R^{4}=T_{p}M \oplus\sigma_{p}, 
\end{equation*}
and for all   tangent vector fields $X,Y$ on $M$,
\begin{equation*}
    (D_{X}Y)_{p}=(\nabla_{X}Y)_{p}+  h(X,Y)_{p},
\end{equation*}
where   $(\nabla_{X}Y)_{p}\in T_pM$ and $h(X,Y)_{p}\in \sigma_p,$ for all $p \in M$.
\medskip

We note that for $p\in M$, there are $\xi_{1},\xi_{2}$ transversal vector fields defined on some neighborhood  $U_p$ such that: $\sigma_q =\mbox{span} \{ \xi_{1}(q), \xi_{2}(q)\}$, $\forall q \in U_p$.
\medskip

Then for tangent vector fields $X,Y$ on $M$ we have:
\begin{align}
D_{X}Y &= \nabla_{X} Y + h^{1}(X,Y)\xi_{1}+ h^{2}(X,Y)\xi_{2}\label{decom1}, \\
D_{X}\xi_{1} &= -S_{1}X +\tau_{1}^{1}(X)\xi_{1}+\tau_{1}^{2}(X)\xi_{2}\label{decom2},  \\
D_{X}\xi_{2} &= -S_{2}X +\tau_{2}^{1}(X)\xi_{1}+\tau_{2}^{2}(X)\xi_{2}\label{decom3},
\end{align}
where $\nabla=\nabla(\sigma)$ is a torsion free affine connection, $h^{1},h^{2}$ are bilinear symmetric forms, $S_{1},S_{2}$ are $(1,1)$ tensor fields, and $\tau_{i}^{j}$ are $1$-forms  on $M$.    We call $\nabla$  the affine connection induced by the transversal plane bundle $\sigma.$

For a transversal vector field $\xi$, we can write
\[ D_{X}\xi = -S_{\xi}X +\nabla^{\bot}_{X}\xi,\]
where $-S_{\xi}X $ is the tangent component of $D_{X}\xi$ and $\nabla^{\bot}_{X}\xi$ is the $\sigma$-component of $D_{X}\xi$. The operator \(-S_{\xi}\) is linear (in fact, a $(1,1)$-tensor field)   and is called shape o\-pe\-ra\-tor.   {We also call $\nabla^{\bot}$ the } affine normal connection induced by the transversal plane bundle $\sigma$.

We can consider the curvature tensor of the normal connection,    {called}\textit{ normal curvature tensor}
\[R_{\nabla^{\bot}}:T_p  M\times T_p M  \times \sigma_p\rightarrow \sigma_p ,\]
given by
\[R_{\nabla^{\bot}}(X,Y)\nu=\nabla^{\bot}_X (\nabla^{\bot}_Y \nu)-\nabla^{\bot}_Y (\nabla^{\bot}_X \nu)-\nabla^{\bot}_{[X,Y]}\nu.\]
Since $\R^4$ has vanishing curvature, we obtain
\[R_{\nabla^{\bot}}(X,Y)\nu=h(X,S_{\nu} Y)- h(Y,S_{\nu} X ).\]

\section{  {The metric of the} transversal vector field.}
In this section, we introduce   {a family of affine metrics} and the affine normal planes:   {the} antisymmetric and symmetric equiaffine  planes. We prove the existence   {and unicity} of these planes and provide algorithms to compute them.

\begin{defn}A surface $M\subset \R^4$ is called \emph{locally convex} at $p$ if there is a hyperplane $\pi$ of $\R^4$ such that $p\in \pi$ and $\pi$ supports $M$ in a neighborhood of $p$. If $\pi$ locally supports $M$ at $p$, it is obvious that $\pi$ is tangent to $M$ at $p$. If it has a non-degenerate contact (i.e., of Morse type), then we say that $M$ is   {\emph{locally strictly convex}} at $p$.
\end{defn}


Let $M\subset \R^4$ be a locally strictly  convex  surface, let $\mathfrak{u}=\{X_1,X_2\}$ be a local tangent frame of a point $p\in M$ and let $\xi$ be a transversal vector field on $M$.

\begin{defn}
We define the symmetric bilinear form $G_{\mathfrak{u}}$ on  $M$ to be
\[G_{\mathfrak{u}}(Y,Z)=[X_1,X_2,D_{Z}Y,\xi].\]
We fix $\xi$ such that $G_{\mathfrak{u}}$ is a positive definite quadratic form, this is possible because $M$   {is} locally strictly convex   {and we call such a $\xi$ a \emph{metric field}.}
We define the \textit{metric of   {the} transversal vector field}, denoted  by $g_{\xi}$, by
\[ g_{\xi}(Y,Z)=\frac{G_{\mathfrak{u}}(Y,Z)}{(\det_{\mathfrak{u}}G_{\mathfrak{u}})^{\frac{1}{4}}},\]
where $\det_{\mathfrak{u}}G_{\mathfrak{u}}= \det (G_{\mathfrak{u}}(X_i,X_j)).$
\end{defn}

\begin{lem} The quadratic form $g_{\xi}$ does not depend on the choice of   {the} tangent frame $\mathfrak{u}$.
\end{lem}
\begin{proof}
Let $\mathfrak{v}=\{Y_1,Y_2\}$ be   {another} local tangent frame on a neighborhood $U$  of $p \in M$, then there exist functions $a, b, c$ and $d$ with $ad- bc \neq0$, defined on $U$ such that $Y_1=a X_1+ b X_2$ and $Y_2=c X_1+d X_2$. Note that
\[G_{\mathfrak{v}}(Y,Z)=[Y_1,Y_2,D_{Z}Y,\xi]=(ad-bc)G_{\mathfrak{u}}(Y,Z).\]
By properties of   {the} determinant,   {it} follows that $\det_{\mathfrak{v}}G_{\mathfrak{v}}=(ad-bc)^2 \det_{\mathfrak{v}}G_{\mathfrak{u}}.$
On the other   {hand}, from a   {simple} computation  $\det_{\mathfrak{v}}G_{\mathfrak{u}}=(ad-bc)^2 \det_{\mathfrak{u}}G_{\mathfrak{u}}$, therefore
\[\textstyle{\det_{ \mathfrak{v}}G_{\mathfrak{v}}=(ad-bc)^4 \det_{\mathfrak{u}}G_{\mathfrak{u}}}.\]
Finally,
\[\frac{G_{\mathfrak{v}}(Y,Z)}{(\det_{\mathfrak{v}}G_{\mathfrak{v}})^{1/4}}=\frac{(ad-bc)G_{\mathfrak{u}}(Y,Z)}{((ad-bc)^4 \det_{\mathfrak{u}}G_{\mathfrak{u}})^{1/4}}=\frac{G_{\mathfrak{u}}(Y,Z)}{(\det_{\mathfrak{u}}G_{\mathfrak{u}})^{1/4}}.\]
\end{proof}


\begin{rem}
Let $\xi\in\R^4$ be a metric field, the definition  of   {the} metric $g_\xi$ depends only on the equivalence class   {of $\xi$} in the quotient space $\R^4/T_p M$,   {which is a 2-dimensional vector space.}
In fact, if  $[\xi]=[\xi']$ then $\xi=\xi'+Z$, with $Z\in T_pM$ therefore $g_\xi=g_{\xi'}$.
  {Thus}, we denote $g_{[\xi]}=g_\xi$.

In this   {way}, the family of metrics $\{g_{[\xi]}\}_{[\xi]\in A}$ is parameterized by an open set of  $\R^4/T_pM$ given by:
\[
A=\{[\xi]\in \R^4/T_pM:\ g_{[\xi]}\text{ is positive definite }\}.
\]
This open subset $A$ is not empty whenever $M$  is strictly  convex in a neighborhood of $p$.
  {It follows from the definition} that the family of metrics $\{g_{[\xi]}\}_{[\xi]\in A}$ is an affine invariant of $M$ that does not depend   {on the chosen}  transversal plane bundle $\sigma$.
\end{rem}

\begin{rem}   {Although the metric $g_{[\xi]}$ does not depend on the chosen transversal plane bundle $\sigma$, once we fix it, we can use it to compute the metric as a second fundamental form. In fact, let $\{\xi_1,\xi_2\}$ be a transversal frame for $\sigma_p$ and
let us denote the second fundamental forms by $h^1(X,Y)$ and $h^2(X,Y)$.} For all $(r,s)\in\R^2$ we denote
\[h_{r,s}(X,Y)=r h^1(X,Y)+s h^2(X,Y)\]
and  we consider the open subset $\tilde A\subset\R^2$ given by:
\[
\tilde A=\{(r,s)\in\R^2:\ h_{r,s}\text{ is positive definite }\}.
\]
We note that, for all  $[\xi]\in A$, there is an unique $(r,s)\in\tilde A$ such that  $g_{[\xi]}=h_{r,s}$. In fact, there is an unique representative $\xi\in \sigma_p$ of $[\xi]$ given by $\xi= b_1 \xi_1+b_2 \xi_2$ and therefore $(r,s)=\lambda(b_2,-b_1)$, where
\[
\lambda=\frac{[X_1,X_2,\xi_1,\xi_2]}{(\det_{\mathfrak{u}}G_{\mathfrak{u}})^{\frac14}}.
\]
\end{rem}

\begin{rem}
Let $M\subset \R^4$ be a locally strictly convex  surface and   {let} $\pi$ be a support hyperplane   {with non-degenerate contact at $p$}. Then there is a transversal vector field $\xi$ such that
\[\pi=\ker\{x \mapsto [Y_1,Y_2,x-p,\xi]\} \]
and $g_{\xi}$ is positive definite, where $\{Y_1,Y_2\}$  is a  tangent local frame on $M$. We say that $\xi$ determines the support hyperplane $\pi$. Note that $\lambda \xi$ determine the same hyperplane $\pi$. So, the support hyperplane $\pi$ determines a family of metrics $\left(g_{\lambda\xi}\right)_{\lambda}$.
\end{rem}

From now   {on}, we fix a metric field $\xi$ and consider a local orthonormal tangent frame  relative to   {the} metric $g_{\xi}$ on $M$,   {that} is, $\mathfrak{u}=\{X_1,X_2\}$   {is a tangent} frame defined on some neighborhood $U$ of $p\in M$ such that
\[g_{\xi}(X_i,X_j)=\delta_{ij}.\]

\begin{thm}\label{vetores-unicos}
Let $M\subset \R^4$ be a  locally strictly convex surface  and $\xi$ a metric field. Let $\mathfrak{u}=\{X_1,X_2\}$ be a local orthonormal  tangent frame of $g_{\xi}$ and   {let} $\sigma$ be an arbitrary transversal plane bundle. Then there exists   {a} unique local   {frame} $\{\xi_1,\xi_2\}$ of $\sigma$, such that
\[[X_1,X_2,\xi_1,\xi_2]=1, \hspace*{1.5cm} h^1(X_1,X_1)=0, \hspace*{1.5cm} -\xi_1\in [\xi],\]
\[\hspace*{0.7cm}h^2(X_1,X_1)=1, \hspace*{1.5cm} h^2(X_1,X_2)=0, \hspace*{1.5cm} h^2(X_2,X_2)=1.\]
\end{thm}

\begin{proof}Let $p$ be a point in $M$ and let $\{\nu_1,\nu_2\}$ be   {any} local   {frame} of $\sigma$ in a neighborhood $U$ of $p$. We can {assume} that $X_1$ and $X_2$ are defined on $U$. Now, we  write
\[[\xi]= \lambda_3 \nu_1+ \lambda_4 \nu_2+ T_p M.\]
Using the notation:\hspace*{0.2cm}  $h^1(X_1,X_1)= a$, $ h^1(X_1,X_2)=b$, $h^1(X_2,X_2)=c$, $h^2(X_1,X_1)= e$, $h^2(X_1,X_2)=f$, $h^2(X_2,X_2)=g$ and $K=[X_1,X_2,\nu_1,\nu_2]$, we compute the bilinear form $G_{\mathfrak{u}}$:
\begin{align*}
  G_{\mathfrak{u}}(X_1,X_1) &= (a \lambda_4- e \lambda_3)K, \\
  G_{\mathfrak{u}}(X_1,X_2) &= (b \lambda_4- f \lambda_3)K, \\
  G_{\mathfrak{u}}(X_2,X_2) &= (c \lambda_4- g \lambda_3)K.
\end{align*}
  {By} using  the change
\[\nu_1=\alpha \xi_1+ \beta \xi_2, \hspace*{1.0cm} \nu_2= \varphi \xi_1+\psi \xi_2,\]
we obtain the affine fundamental forms in the new   {frame} $\{\xi_1,\xi_2\}$:
\[\overline{h}^1(X_1,X_1)=  \alpha a + \varphi e,  \hspace*{0.5cm} \overline{h}^1(X_1,X_2)=\alpha b + \varphi f,  \hspace{0.5cm} \overline{h}^1(X_2,X_2)=\alpha c + \varphi g, \]
\[\overline{h}^2(X_1,X_1)= \beta a + \psi e,  \hspace*{0.5cm} \overline{h}^2(X_1,X_2)=\beta b + \psi f, \hspace{0.5cm} \overline{h}^2(X_2,X_2)=\beta c + \psi g.\]
  {By} solving the following system:
\begin{align*}
  1 &= \beta a + \psi e,\\
  0 &= \beta \lambda_3+ \psi \lambda_4,
\end{align*}
we obtain
\[\beta=\frac{\lambda_4}{a \lambda_4- e \lambda_3} ,\hspace*{1.0cm}\psi=\frac{- \lambda_3}{a \lambda_4- e \lambda_3}.\]
  {We substitute} $\beta$ and $\psi$ in  $\overline{h}^2(X_i,X_j)$   {and} we prove that $\overline{h}^2(X_i,X_j)=\delta_{ij}.$ In fact,
\begin{align*}
\overline{h}^2(X_1,X_2)&=\beta b + \psi f =(\frac{ \lambda_4}{a \lambda_4- e \lambda_3})b + (\frac{-  \lambda_3}{a \lambda_4- e \lambda_3})f=\frac{ G_{\mathfrak{u}}(X_1,X_2)}{(a \lambda_4- e \lambda_3)K}\\
&=\frac{G_{\mathfrak{u}}(X_1,X_2)}{G_{\mathfrak{u}}(X_1,X_1)}=\frac{G_{\mathfrak{u}}(X_1,X_2)/(\det_{\mathfrak{u}}G_{\mathfrak{u}})^{1/4}}{G_{\mathfrak{u}}(X_1,X_1)/(det_{\mathfrak{u}}G_{\mathfrak{u}})^{1/4}}=\frac{g_{\xi}(X_1,X_2)}{g_{\xi}(X_1,X_1)}=0. \end{align*}

From the equation $0=\overline{h}^1(X_1,X_1)= \alpha a + \varphi e$ we can write $\alpha= R e$ and $\varphi= - R a$, therefore
\begin{align*}
[X_1,X_2, \nu_1,\nu_2]&=[X_1,X_2, \xi_1,\xi_2](\alpha \psi- \beta\varphi)=(\alpha \psi- \beta\varphi)\\&=( (R e)\psi- \beta( -R a))=R,
\end{align*}
we conclude $R=K,$ $\alpha= K e$ and $\varphi= - K a.$

It   {only} remains to prove that $[\xi]=-[\xi_1]$.
First we note that $G_{\mathfrak{u}}(X_1,X_2)=0$, because $\{X_1,X_2\}$ is a orthonormal tangent frame relative to $g_{\xi}$.
Moreover,
\[\textstyle{(\det_{\mathfrak{u}}G_{\mathfrak{u}})^{1/2}}=\frac{\det_{\mathfrak{u}}G_{\mathfrak{u}}}{(\det_{\mathfrak{u}}G_{\mathfrak{u}})^{1/2}}=\frac{G_{\mathfrak{u}}(X_1,X_1)}{(\det_{\mathfrak{u}}G_{\mathfrak{u}})^{1/4}}\frac{G_{\mathfrak{u}}(X_2,X_2)}{(\det_{\mathfrak{u}}G_{\mathfrak{u}})^{1/4}}=1.\]
It follows that
$$
\textstyle{\lambda_3 \alpha+\lambda_4 \varphi=\lambda_3 Ke -\lambda_4 K a=K(\lambda_3 e- \lambda_4 a)=-(\det_{\mathfrak{u}} G_{\mathfrak{u}})^{1/4}=-1.}
$$
Finally, we compute   {$[\xi]$:}
\begin{align*}
[\xi]&=\lambda_3 \nu_1+ \lambda_4 \nu_2+  T_p M\\
&=\lambda_3 (\alpha \xi_1+\beta \xi_2)+ \lambda_4 (\varphi \xi_1+\psi \xi_2)+T_p M\\
&=\underbrace{(\lambda_3 \alpha+\lambda_4 \varphi)}_{-1}\xi_1+ \underbrace{(\lambda_3 \beta+ \lambda_4 \psi)}_0\xi_2+T_pM.
\end{align*}
\end{proof}

\begin{lem}\label{lema-cambio de referencia tangente}Let $M \subset \R^4$ be a locally strictly  convex surface  and $\xi$ a metric field. Let $\mathfrak{u}=\{X_1,X_2\}$ and $\mathfrak{v}=\{Y_1,Y_2\}$ be two orthonormal frames and let $\sigma$ a transversal plane bundle. So we can write
\begin{align}
  Y_1 &= \cos \theta X_1+ \sin \theta X_2, \label{eqcambio1}  \\
  Y_2 &= \epsilon (-\sin \theta X_1+\cos \theta X_2).\label{eqcambio2}
\end{align}
where $\epsilon =\pm 1$. If we denote by $\{\xi_1,\xi_2\}$ (resp. $\{\overline{\xi}_1,\overline{\xi}_2\}$) the   {frame of theorem \ref{vetores-unicos}} corresponding to $\mathfrak{u}$ (resp. $\mathfrak{v}$), then
\begin{align*}
  \xi_1 &= \overline{\xi}_1, \\
  \xi_2 &= -(\sin 2\theta  h^1(X_1,X_2)+\sin^2\theta h^1(X_2,X_2))\overline{\xi}_1+\overline{\xi}_2,
\end{align*}
and also
\begin{align*}
  2\overline{h}^1(Y_1,Y_2) &= \epsilon(2 \cos2\theta h^1(X_1,X_2)+\sin2\theta h^1(X_2,X_2)), \\
  \overline{h}^1(Y_2,Y_2)  &=  \cos2\theta h^1(X_2,X_2)-2 \sin2\theta h^1(X_1,X_2),\\
4 \overline{h}^1(Y_1,Y_2)^2&+ \overline{h}^1(Y_2,Y_2)=4 h^1(X_1,X_2)^2+ h^1(X_2,X_2)^2.
\end{align*}
\end{lem}
\begin{proof} From theorem \ref{vetores-unicos},   {we have $[\xi_1]=-[\xi]=[\overline{\xi}_1]$.} Since $\xi_1$ and $\overline{\xi}_1$ belong to {the} same transversal plane we conclude that $\xi_1=\overline{\xi}_1$.
We compute now the affine connection in the different references to compare the references.   {By
using the frame $\{\overline{\xi}_1,\overline{\xi}_2\}$, it} follows from theorem \ref{vetores-unicos} that
\[D_{Y_1}Y_1=\nabla_{Y_1}Y_1+\overline{\xi}_2,\]
and   {by} using the reference $\{\xi_1,\xi_2\}$ and   {the} equation \eqref{decom1},
\[D_{Y_1}Y_1=\nabla_{Y_1}Y_1+h^1(Y_1,Y_1)\xi_1+h^2(Y_1,Y_1)\xi_2.\]
  {Hence, } $\overline{\xi}_2 = h^1(Y_1,Y_1)\xi_1+h^2(Y_1,Y_1)\xi_2$,   and from equation \eqref{eqcambio1} we obtain:
\begin{equation*}
   \overline{\xi}_2 = (\sin 2 \theta h^1(X_1,X_2)+\sin^2 \theta h^1(X_2,X_2))\xi_1+\xi_2.
\end{equation*}

Analogously, by   {comparing} $D_{Y_1}Y_2$ (and $D_{Y_2}Y_2$) in the   {two frames,}  we obtain:
\begin{equation*}
    \overline{h}^1(Y_1,Y_2)=\cos 2 \theta h^1(X_1,X_2)+ \sin \theta \cos \theta h^1(X_2,X_2),
\end{equation*}
\begin{equation*}
    \overline{h}^1(Y_2,Y_2) =  \cos2\theta h^1(X_2,X_2)-2 \sin2\theta h^1(X_1,X_2).
\end{equation*}
  {The last equality follows by direct computation.}
\end{proof}

Let $M\subset \R^4$ be a  locally strictly convex surface  and let $\xi$ be a metric field. Let $\mathfrak{u}=\{X_1,X_2\}$ be a local orthonormal  tangent frame. If we denote the corresponding transversal vector fields obtained by theorem \ref{vetores-unicos} by $\xi_1$ and $\xi_2$, we define the metric $g_{\mathfrak{u}}^{\bot}$ by setting
\begin{align*}
g_{\mathfrak u}^{\bot}(\xi_1,\xi_1)&= 1,\\
g_{\mathfrak u}^{\bot}(\xi_1,\xi_2)&=-\frac{1}{2}h^1(X_2,X_2),\\
g_{\mathfrak u}^{\bot}(\xi_2,\xi_2)&=4 h^1(X_1,X_2)^2+\frac{5}{4}h^1(X_2,X_2)^2.
\end{align*}
and extending it linearly on $\sigma.$

\begin{lem}\label{gbot} Take $\mathfrak{u}=\{X_1,X_2\}$, $\mathfrak{v}=\{Y_1,Y_2\}$, $\xi_1$, $\xi_2$, $\overline{\xi}_1$ and $\overline{\xi}_2$ as in  lemma \ref{lema-cambio de referencia tangente} and theorem \ref{vetores-unicos}. Then
\[g^{\bot}_{\mathfrak{u}}(\xi,\eta)=g^{\bot}_{\mathfrak{v}}(\xi,\eta).\]
\end{lem}
\begin{proof}It is enough to show that equality occurs on  the frame $\{\xi_1,\xi_2\}$.
We have $g_{\mathfrak{v}}^{\bot}(\xi_1,\xi_1) =g_{\mathfrak{v}}^{\bot}(\overline \xi_1,\overline \xi_1)=1$, but also
\begin{align*}
g_{\mathfrak{v}}^{\bot}(\xi_1,\xi_2) &=g_{\mathfrak{v}}^{\bot}(\overline \xi_1,h^1(Y_1,Y_1) \overline \xi_1+ \overline \xi_2)\\
&=h^1(Y_1,Y_1)-\frac 1 2 \overline h^1(Y_2,Y_2)
=-\frac 1 2 h^1(X_2,X_2),
\end{align*}
and finally,
\begin{align*}
g_{\mathfrak{v}}^{\bot}(\xi_2,\xi_2) &=g_{\mathfrak{v}}^{\bot}(h^1(Y_1,Y_1) \overline \xi_1+ \overline \xi_2,h^1(Y_1,Y_1) \overline \xi_1+ \overline \xi_2)\\
&=h^1(Y_1,Y_1)^2+ 2 h^1(Y_1,Y_1)g_{\mathfrak v}^{\bot}(\overline \xi_1, \overline \xi_2)+ g_{\mathfrak v}^{\bot}(\overline \xi_2, \overline \xi_2)\\
&=h^1(Y_1,Y_1)^2-  h^1(Y_1,Y_1) \overline h^1(Y_2,Y_2)+4 \overline h^1(Y_1,Y_2)^2+ \frac 5 4 \overline h^1(Y_2,Y_2)^2\\
&=(h^1(Y_1,Y_1)- \frac 1 2 \overline h^1(Y_2,Y_2))^2+4\overline h^1(Y_1,Y_2)^2+\overline h^1(Y_2,Y_2)^2\\
&=\frac 14 h^1(X_2,X_2)^2+4h^1(X_1,X2)^2+h^1(X_2,X_2)^2.
\end{align*}
\end{proof}

By lemma \ref{gbot},  $g_{\mathfrak u}^{\bot}$ is independent of the choice of the tangent frame $\mathfrak u$, we denote it by $g_{\xi}^{\bot}.$


\section{Asymptotic directions, affine binormals and inflections}
In this section, we introduce the concepts of asymptotic directions, affine binormals and inflections. These concepts are well known in the case of a surface immersed in Euclidean space (see for instance \cite{MRR,NR}). We show how to adapt all these definitions to the context of the affine differential geometry.

\medskip
Let $M\subset\R^4$ be an immersed surface with a transversal plane bundle $\sigma$. We denote by $\sigma^*$ the conormal, that is, the dual vector bundle of $\sigma$. For any $p\in M$ and for any conormal vector $\lambda\in\sigma^*_p$, we define the second fundamental form along $\lambda$ as:
$$
h_\lambda(X,Y)=\lambda(h(X,Y)), \ \forall X,Y\in T_pM.
$$

\begin{defn} We say that a non zero $\lambda\in\sigma^*_p$ is an \emph{affine binormal} at $p$ if $h_\lambda$ is degenerate, that is, if there is a non zero tangent vector $X\in T_pM$ such that
$$
h_\lambda(X,Y)=0,\ \forall Y\in T_pM.
$$
Moreover, in such a case, we say that $X$ is an \textit{asymptotic direction} at $p$ associated with the affine binormal $\lambda$.
\end{defn}

The concepts of asymptotic and affine binormals are related to the so-called generalized eigenvalue problem.
Let $A$, $B$ be two $n\times n$ matrices. A pair $(p, q)\in \R^{2}- \{0\}$  is a \textit{generalized eigenvalue} of $(A,B)$ if
\[\det(pA + qB) = 0.\]
Analogously,   {$x\in \R^{n}-\{0\}$} is a \textit{generalized eigenvector} associated with the generalized eigenvalue $(p,q)$ if \[(pA + qB)x = 0.\]

In our case, given a point $p\in M$ we fix $\mathfrak{u}=\{X_1,X_2\}$ any tangent frame of $T_pM$, $\{\xi_1,\xi_2\}$ any transversal frame of $\sigma_p$ and $\{\lambda_1,\lambda_2\}$ the corresponding dual frame of $\sigma^*_p$. We denote by $A=(h^1(X_i,X_j))$ and $B=(h^2(X_i,X_j))$ the coefficient matrices of the second fundamental forms $h^1,h^2$ respectively. The proof of the following lemma is straightforward from the definitions.

\begin{lem}\label{eigenvalue} With the above notation, $X=u_{1}X_{1}+u_{2}X_{2}\in T_pM$ is an asymptotic direction associated with the affine binormal $\lambda=r \lambda_{1}+ s \lambda_{2} \in \sigma^*_p$ if and only if $u=(u_{1},u_{2})$ is a generalized eigenvector of $(A,B)$ associated with the generalized eigenvalue $(r,s)$.
\end{lem}

It follows from lemma \ref{eigenvalue} that the affine binormals are determined by the solutions of the quadratic equation $\det(rA+sB)=0$, so we can have either 2, 1 or 0 affine binormal directions. When $M$ is locally strictly convex, we always have at least one affine binormal.

\begin{cor} Let $M\subset\R^4$ be a locally strictly convex surface with a transversal bundle $\sigma$. At any point $p\in M$, either:
\begin{enumerate}
\item there exist exactly two affine binormal directions and two asymptotic directions (one for each binormal), or
\item there exists exactly one affine binormal direction and any tangent direction is asymptotic.
\end{enumerate}
\end{cor}

\begin{proof}
We choose any metric field $\xi$ on $M$ and consider $\mathfrak{u}=\{X_1,X_2\}$ an orthonormal tangent frame and $\{\xi_1,\xi_2\}$ the associated transversal frame given by theorem \ref{vetores-unicos}. The coefficient matrices of the second fundamental forms are
\[
A=\left(
\begin{array}{cc}
0  & b    \\
b  & c
\end{array}
\right),
\quad
B=\left(
\begin{array}{cc}
1  & 0    \\
0  & 1
\end{array}
\right),
\]
where $b=h^1(X_1,X_2)$ and $c=h^1(X_2,X_2)$. By lemma \ref{eigenvalue}, the asymptotic and affine binormal directions are given in terms of the solutions of the homogeneous linear system:
\[
\left(
\begin{array}{cc}
s  & rb    \\
rb  & rc+s
\end{array}
\right)
\left(
\begin{array}{c}
u_1\\u_2\end{array}
\right)=
\left(
\begin{array}{c}
0\\0\end{array}
\right).
\]
The affine binormal directions are given by the roots of the determinant $s^2+crs-b^2r^2=0$. Since $(r,s)\ne(0,0)$, we can assume $r\ne0$ and normalize to $r=1$, so,
$$s=\frac{-c\pm\sqrt{c^2+4b^2}}2.$$
If $(b,c)\ne(0,0)$, we have two distinct solutions and one asymptotic direction $(u_1,u_2)$ for each one of them. Otherwise, if $(b,c)=(0,0)$, then $s=0$ and all the directions $(u_1,u_2)$ are asymptotic.
\end{proof}

\begin{defn} We say that a point $p\in M$ is an \emph{inflection} if all the tangent directions at $p$ are asymptotic, that is, if there is a non zero $\lambda\in\sigma^*$ such that $h_\lambda=0$.
\end{defn}

With the notation of lemma \ref{eigenvalue}, $p$ is an inflection if and only if the matrices $A,B$ are collinear. In the case that $M$ is locally strictly convex, we fix a metric field $\xi$ and take an orthonormal tangent frame $\mathfrak{u}=\{X_1,X_2\}$ and a transversal frame $\{\xi_1,\xi_2\}$ as in theorem \ref{vetores-unicos}. Then $p$ is an inflection if and only if $h^1(X_1,X_2)=h^1(X_2,X_2)=0$.

\medskip

We can also use the lemma \ref{eigenvalue} in order to obtain the differential equation of the asymptotic lines of a surface. By definition, an asymptotic line is an integral curve of the field of asymptotic directions, that is, it is a curve whose tangent at any point is asymptotic.

\begin{thm}
With the notation of lemma \ref{eigenvalue}, the differential equation for the asymptotic lines of $M$ is:
\[ \left|
  \begin{array}{ccc}
    dv^{2} & -dv du & du^{2} \\
    a & b & c \\
    e & f & g \\
  \end{array}
\right|=0,\]
where $A=\left(
           \begin{array}{cc}
             a & b \\
             b & c \\
           \end{array}
         \right)
$ and $B=\left(
         \begin{array}{cc}
           e & f \\
           f & g \\
         \end{array}
       \right).
$
\end{thm}
\begin{proof}
We just eliminate $(r,s)$ in the linear system $(rA+sB)u=0$, where $u=(du,dv)$.
\end{proof}

Another important fact is that we can characterize the asymptotic directions and affine binormals in terms of the singularities of the height functions. Given the direct sum $\R^4=T_pM\oplus\sigma_p$, we denote by $p_1:\R^4\to T_pM$ and $p_2:\R^4\to\sigma_p$ the two associated linear projections. Then, for each $\lambda\in\sigma_p^*$, we define the height function $H_\lambda:M\to\R$ by
$$
H_\lambda(x)=\lambda(p_2(x)).
$$

\begin{prop}\label{height-function} Let $\lambda\in\sigma_p^*$ be a non zero conormal vector of $M$, then:
 \begin{enumerate}
\item $H_\lambda$ has always a singularity at $p$;
\item $\lambda$ is an affine binormal if and only if $H_\lambda$ has a degenerate singularity at $p$;
\item $X\in T_pM$ is an asymptotic direction associated with $\lambda$ if and only if $X$ belongs to the kernel of the Hessian of $H_\lambda$ at $p$;
\item $p$ is an inflection if and only if there exists a non zero $\lambda\in \sigma_p^*$ such that $H_\lambda$ has a corank 2 singularity at $p$.
\end{enumerate}
\end{prop}

\begin{proof}
The differential of $H_\lambda$ at $p$ is always zero and we have (1):
$$
d(H_\lambda)_p(X)=\lambda(p_2(X))=0,\ \forall X\in T_pM.
$$
But the Hessian of $H_\lambda$ at $p$ is precisely the second fundamental form $h_\lambda$:
$$
d^2(H_\lambda)_p(X,Y)=\lambda(p_2(D_XY))=\lambda(h(X,Y))=h_\lambda(X,Y),\ \forall X,Y\in T_pM.
$$
Then, (2), (3) and (4)  follow directly from the definitions of affine binormal, asymptotic direction and inflection.
\end{proof}

The results of proposition \ref{height-function} can be easily restated in terms of contacts with hyperplanes.

\begin{defn} We say that $\pi$ is an \emph{osculating hyperplane} of $M$ at $p$ if it is tangent to $M$ at $p$ and it has a degenerate contact with $M$ at $p$. If $H:\R^4\to\R$ is any linear function such that $\pi$ is given by the equation $H(x-p)=0$, then we say that $X\in T_pM$ is a \emph{contact direction} if it belongs to the kernel of the Hessian of $H|_M$.
\end{defn}

\begin{cor}
A tangent vector $X\in T_pM$ is an asymptotic direction if and only if it is a contact direction of some osculating hyperplane. In particular, the asymptotic directions (and hence the inflections) of $M$ are affine invariant, that is, they do not depend on the choice of the transversal plane bundle $\sigma$.
\end{cor}

\begin{proof} If $X$ is an asymptotic direction, then there exists an affine binormal $\lambda\in\sigma^*_p$ associated with $X$. We define $\pi$ as the hyperplane passing through $p$ and parallel to $T_pM\oplus\ker\lambda$. We can take $H:\R^4\to\R$ given by $H(x)=\lambda(p_2(x))$ so that $H(x-p)=0$ is a defining equation of $\pi$ and $H|_M=H_\lambda$. By proposition \ref{height-function}, $\pi$ is an osculating hyperplane and $X$ is a contact direction.

Conversely, assume that $\pi$ is an osculating hyperplane and $X$ is a contact direction. Let $H:\R^4\to\R$ be any linear function such that $H(x-p)=0$ is a defining equation of $\pi$. We take now $\lambda=H|_{\sigma_p}\in\sigma_p^*$, then
$$
H|_M(x)=H(x)=H(p_1(x)+p_2(x))=H(p_2(x))=\lambda(p_2(x))=H_\lambda(x),
$$
for all $x\in M$. Again by proposition \ref{height-function}, $\lambda$ is an affine binormal with associated asymptotic direction $X$.
\end{proof}

\begin{rem} If $M$ is locally strictly convex and $\xi$ is a metric field, then we can use the transversal metric $g_\xi^\bot$ in order to define binormal directions also in $\sigma$ instead of $\sigma^*$. In fact, for each $\nu\in\sigma_p$ we have a well defined second fundamental form along $\nu$:
$$
h_\nu(X,Y)=g^\bot_\xi(h(X,Y),\nu),\ \forall X,Y\in T_pM.
$$
We say that $\nu$ is binormal if $h_\nu$ is degenerate.
\end{rem}

\begin{rem} It is not difficult to see that if $M$ is non-degenerate in the sense of Nomizu-Vranken \cite{N-V}, then the asymptotic directions of $M$ at $p$ are exactly the null directions of the Burstin-Mayer affine metric (which is indefinite in the case that $M$ is locally strictly convex). This gives another proof of the fact that the asymptotic directions are affine invariant.
\end{rem}

We finish this section with the next result, which gives the relation between the transversal frames of theorem \ref{vetores-unicos} in the case we change the transversal plane. This lemma will be very useful in the next section.

\color{black}

\begin{lem}\label{remark-cambiodebase}   {Let $M \subset \R^4$ be a locally strictly convex surface without inflections. Let $\xi$ be a metric field and $\mathfrak{u}=\{X_1,X_2\}$ a local orthonormal  tangent frame. Let  $\sigma$ and $\overline{\sigma}$ be two transversal plane bundles. We denote by $\{\xi_1,\xi_2\}$ and $\{\overline{\xi}_1,\overline{\xi}_2\}$ the transversal frames obtained from theorem \ref{vetores-unicos} for  $\sigma$ and $\overline{\sigma}$, respectively.} Then there are $Z_1$ and $Z_2$ tangent vector fields on $M$ such that
\begin{align*}
  \overline{\xi}_1 &= \xi_1+ Z_1, \\
  \overline{\xi}_2 &= \xi_2+ Z_2.
\end{align*}
\end{lem}
\begin{proof}We suppose that
\[\overline{\xi}_1 =\phi \xi_1 + \psi \xi_2 + Z_1, \hspace*{1.0cm} \overline{\xi}_2 =\rho \xi_1 + \beta \xi_2 + Z_2.\]
By theorem \ref{vetores-unicos} we have $[X_1,X_2,\overline{\xi}_1,\overline{\xi}_2]=1$,   {which implies}  $\phi \beta- \psi \rho=1$. We denote by $\overline{h}^1$ and $\overline{h}^2$ the affine fundamental   {forms of the frame} $\{\overline{\xi}_1,\overline{\xi}_2\}$.   {We} note that
\begin{align*}
  \overline{h}^1(X_1,X_1) &= \phi h^1(X_1,X_1)+ \rho h^2(X_1,X_1), \\
  \overline{h}^2(X_1,X_1) &= \psi h^1(X_1,X_1)+ \beta h^2(X_1,X_1).
\end{align*}
  {Again by  theorem \ref{vetores-unicos},} $\overline{h}^1(X_1,X_1)=0=h^1(X_1,X_1)$ and $\overline{h}^2(X_1,X_1)=1=h^2(X_1,X_1)$, therefore $\rho=0$ {and} $\beta=1$. Since $\beta=1$,    {we have}
\begin{align*}
  \overline{h}^2(X_1,X_2) &= \psi h^1(X_1,X_2)+  h^2(X_1,X_2), \\
  \overline{h}^2(X_2,X_2) &= \psi h^1(X_2,X_2)+  h^2(X_2,X_2),
\end{align*}
and by theorem \ref{vetores-unicos}, $\psi h^1(X_1,X_2)=\psi h^1(X_2,X_2)$=0. Note that   {there are no affine inflections, so} $\psi=0$. Finally, from $\phi \beta- \psi \rho=1$ it follows $\phi=1$.
\end{proof}

\section{The equiaffine transversal plane bundles}

Nomizu and Vrancken in \cite{N-V} had defined   {the concept of} equiaffine plane as   {a} transversal plane bundle $\sigma$ such that the affine connection induced by $\sigma$, $\nabla=\nabla(\sigma)$ satisfies $\nabla \omega_g=0$ where $\omega_g$ is the   {metric volume form for the Burstin-Mayer affine metric $g$:
\[\omega_{g}(X,Y)=\sqrt{|g(X,X)g(Y,Y)-g(X,Y)^2|},\]
where $\{X,Y\}$ is any positively oriented basis of $T_pM$.}
In our   {case}, we consider the same definition,   {but we use}  the metric of   {the} transversal vector field   {$g_\xi$ instead of the Burstin-Mayer affine metric}. This definition is based   {on} the compatibility between  the volume form and   {the} affine connection.

\medskip

Let $M \subset \R^4$ be a locally strictly convex surface   {and} $\xi$ a metric field and $g=g_\xi$ the metric of the transversal field $\xi$.

\begin{defn}
We say a transversal plane bundle $\sigma$ is    {\textit{equiaffine}} if the  connection $\nabla=\nabla(\sigma)$ induced by $\sigma$ satisfies  $\nabla \omega_{g}=0.$
\end{defn}

  {If $\mathfrak{u}=\{X_1,X_2\}$ is a local orthonormal  tangent frame and $\{\xi_1,\xi_2\}$ is the transversal frame given by theorem \ref{vetores-unicos}, then
$\omega_g=\theta$, where $\theta$ is the volume form induced by the determinant:
$$
\theta(X,Y)=[X,Y,\xi_1,\xi_2],\ \forall X,Y\in T_pM.
$$
This is because $\omega_g(X_i,X_j)=\theta(X_i,X_j)$, $\forall i,j$. Moreover, by using $\theta$ instead of $\omega_g$, it is easy to see that
$\sigma$ is equiaffine if and only if
\begin{align}
   B_1 &:= (\nabla g)(X_1,X_1,X_1)+(\nabla g)(X_1,X_2,X_2)=0, \label{eq-planoequiafim1}\\
   B_2 &:= (\nabla g)(X_2,X_1,X_1)+(\nabla g)(X_2,X_2,X_2)=0\label{eq-planoequiafim2}.
\end{align}
}

\begin{lem}\label{lema-eq-equiafim}
  {Let $M \subset \R^4$ be a  locally strictly convex surface and $\xi$ a metric field. If $p\in M$ is not an inflection, then there exists an equiaffine plane bundle $\sigma$ defined on a neighborhood of $p$.}
\end{lem}
\begin{proof} Let $\mathfrak{u}=\{X_1,X_2\}$ be   {an} orthonormal tangent frame defined on some neighborhood $U$ of   {$p$}. Let $\overline{\sigma}$ be a transversal plane bundle defined also on $U$  and $\{\overline{\xi}_1,\overline{\xi}_2\}$ the local   {basis of} $\overline{\sigma}$ obtained by theorem \ref{vetores-unicos}.  Now we   {want to construct a new equiaffine plane bundle} $\sigma$  defined on $U$, with local   {basis} $\{\xi_1,\xi_2\}$ obtained also by theorem \ref{vetores-unicos}. By   {lema} \ref{remark-cambiodebase}, we have
\[\xi_1=\overline{\xi}_1- Z_1, \hspace*{1.5cm} \xi_2=\overline{\xi}_2- Z_2,\]
where $Z_1$ and $Z_2$ are tangent vector fields. We denote the connection induced by $\sigma$ (resp. $\overline{\sigma}$) by $\nabla$ (resp. $\overline{\nabla}$).  On the other hand, by a simple calculation we obtain
\begin{align*}
  \overline{B}_1 &= B_1+2 g(Z_2,X_1)+2 h^1(X_1,X_2)g(Z_1,X_2), \\
  \overline{B}_2 &= B_2+2h^1(X_1,X_2)g(Z_1,X_1)+2h^1(X_2,X_2)g(Z_1,X_2)+2g(Z_2,X_2).
\end{align*}
  {Note that} $\sigma$ is    {equiaffine if and only if} $B_1=B_2=0$.    {By} writing $Z_1=a X_1+ b X_2$ and $Z_2=c X_1 + d X_2$,    {this is equivalent to}
\begin{align*}
  \overline{B}_{1} &= 2 c + 2 b h^1(X_1,X_2), \\
  \overline{B}_{2} &= 2 d + 2 a h^1(X_1,X_2)+ 2 b h^1(X_2,X_2).
\end{align*}
The lemma follows since the system above   {has a} solution.   {For instance, set} $a=b=0$, $c=\frac{\overline{B}_1}{2}$ and $d=\frac{\overline{B}_2}{2}$.
\end{proof}

  {We define now the affine normal plane bundles. Our} construction is based on the ideas developed in \cite{N-V}. Because there are many equiaffine plane bundles (lemma \ref{lema-eq-equiafim}), we give conditions to choose some special types among them.

\begin{defn}   {Let $M \subset \R^4$ be a  locally strictly convex surface, $\xi$ a metric field and $\mathfrak{u}=\{X_1,X_2\}$ an orthonormal tangent frame.}
We say that an equiaffine plane bundle $\sigma$ is:
\begin{itemize}
  \item \emph{symmetric}, if
\begin{align*}
     D_1 &=(\nabla g)(X_2,X_1,X_1)-(\nabla g)(X_1,X_2,X_1) = 0,\\
     D_2 &=(\nabla g)(X_1,X_2,X_2)-(\nabla g)(X_2,X_1,X_2) = 0,
\end{align*}
  \item \emph{antisymmetric}, if
\begin{align*}
   C_1 &= (\nabla g)(X_2,X_1,X_1)+(\nabla g)(X_1,X_2,X_1)=0, \\
   C_2 &= (\nabla g)(X_1,X_2,X_2)+(\nabla g)(X_2,X_1,X_2)=0.
\end{align*}
\end{itemize}
\end{defn}

\begin{thm}\label{teo-normalantisymmafim}   {Let $M \subset \R^4$ be a locally strictly convex surface and $\xi$ be a metric field. If $p\in M$ is not an inflection, then there exists a unique antisymmetric equiaffine plane bundle $\sigma$ defined on a neighborhood of $p$.}
\end{thm}
\begin{proof}   {
Let $\mathfrak{u}=\{X_1,X_2\}$ be an orthonormal tangent frame on a neighborhood $U$ of $p$. We consider $\overline{\sigma}$ an equiaffine plane bundle defined on $U$ and $\{\overline{\xi}_1,\overline{\xi}_2\}$ the local basis of
$\overline{\sigma}$ obtained by theorem \ref{vetores-unicos}.}

  {Now, we want to construct a new antisymmetric equiaffine plane bundle $\sigma$ defined on $U$.}  Again by theorem \ref{vetores-unicos} we have $\{\xi_1,\xi_2\}$ a basis of $\sigma$, and by lemma \ref{remark-cambiodebase} we write
\[\xi_1=\overline{\xi}_1- Z_1, \hspace*{1.5cm} \xi_2=\overline{\xi}_2- Z_2,\]
where $Z_1$ and $Z_2$ are tangent vector fields. We denote by $\nabla$ (resp. $\overline{\nabla}$) the affine  connection induced by $\sigma$ (resp. $\overline{\sigma}$).  We compute $\overline{C}_1$ and $\overline{C}_2$
\begin{align*}
  \overline{C}_1 &= C_1+3 h^1(X_1,X_2)g(Z_1,X_1)+ g(Z_2,X_2), \\
  \overline{C}_2 &= C_2+3 h^1(X_1,X_2)g(Z_1,X_2)+h^1(X_2,X_2)g(Z_1,X_1)+ g(X_1,Z_2).
\end{align*}
Since $\sigma$ is antisymmetric, $C_1=C_2=0$ and writing  $Z_1=a X_1+ b X_2$ and $Z_2= c X_1+ d X_2$ we obtain the system
\begin{align*}
  0 &=  c +  b \overline{h}^1(X_1,X_2), \\
  0 &=  d +  a \overline{h}^1(X_1,X_2)+  b \overline{h}^1(X_2,X_2), \\
  \overline{C}_1 &= 3 a \overline{h}^1(X_1,X_2)+ d, \\
  \overline{C}_2 &= 3b \overline{h}^1(X_1,X_2)+ a \overline{h}^1(X_2,X_2)+ c.
\end{align*}
The  determinant of the linear system above in the variables $a, b, c, d$ is
$$4  \overline{h}^1(X_1,X_2)^2+  \overline{h}^1(X_2,X_2)^2\ne0,$$ since $p$ is not an inflection. Therefore, the system has a unique solution.
\end{proof}

\begin{thm}\label{teo-normalsymmafim}
  {With the same hypothesis as in theorem \ref{teo-normalantisymmafim}, there exists a unique symmetric equiaffine plane bundle $\sigma$ defined on a neighborhood of $p$.}
\end{thm}

\begin{proof}   {We follow the same arguments as in the proof of theorem \ref{teo-normalantisymmafim}, but instead of $\overline{C}_1$ and $\overline{C}_2$, we compute $\overline{D}_1$ and $\overline{D}_2$:
\begin{align*}
  \overline{D}_1 &= D_1+ h^1(X_1,X_2)g(Z_1,X_1)- g(Z_2,X_2), \\
  \overline{D}_2 &= D_2+ h^1(X_1,X_2)g(Z_1,X_2)-h^1(X_2,X_2)g(Z_1,X_1)- g(X_1,Z_2).
\end{align*}
By writing  $Z_1=a X_1+ b X_2$ and $Z_2= c X_1+ d X_2$ we obtain again a linear  system in $a,b,c,d$:
\begin{align*}
  0 &=  c +  b \overline{h}^1(X_1,X_2), \\
  0 &=  d +  a \overline{h}^1(X_1,X_2)+  b \overline{h}^1(X_2,X_2), \\
  \overline{D}_1 &=  a \overline{h}^1(X_1,X_2)- d, \\
  \overline{D}_2 &=  b \overline{h}^1(X_1,X_2)- a \overline{h}^1(X_2,X_2)- c,
\end{align*}
whose determinant is again
$4  \overline{h}^1(X_1,X_2)^2+  \overline{h}^1(X_2,X_2)^2\ne0.
$}
\end{proof}

\begin{rem}\label{obs-normalafin} The lemma \ref{lema-eq-equiafim} and  theorem \ref{teo-normalantisymmafim} provide an algorithm to calculate a {basis} $\{\xi_1,\xi_2\}$ of the antisymmetric equiaffine  plane bundle: Let $\sigma$ an arbitrary transversal plane and $\nu_1$, $\nu_2$, $h^1(X_1,X_2)$, $h^1(X_2,X_2)$ are obtained by theorem \ref{vetores-unicos}. Denote by $\nabla$ the affine connection induced by $\sigma$ and we write
\begin{align*}
  \nabla_{X_1}X_1 &= a_1 X_1+ a_2 X_2, \\
  \nabla_{X_1}X_2 &= a_3 X_1+ a_4 X_2,\\
  \nabla_{X_2}X_1 &= a_5 X_1+ a_6 X_2, \\
  \nabla_{X_2}X_2 &= a_7 X_1+ a_8 X_2,
\end{align*}
then:
\begin{align*}
  \xi_1 &= \nu_1- a X_1- b X_2, \\
  \xi_2 &= \nu_2- c X_1- d X_2,
\end{align*}
where:
\begin{align*}
  a &= \frac{-2(a_2+a_3+a_5-a_8) h^1(X_1,X_2)-(a_4+a_6+a_7-a_1) h^1(X_2,X_2)}{4 h^1(X_1,X_2)^2+h^1(X_2,X_2)^2}, \\
  b &= \frac{-2(a_4+a_6+a_7-a_1) h^1(X_1,X_2)+(a_2+a_3+a_5-a_8) h^1(X_2,X_2)}{4 h^1(X_1,X_2)^2+h^1(X_2,X_2)^2}, \\
  c &= -(a_1+a_4+b h^1(X_1,X_2)), \\
  d &= -(a_5+a_8+a h^1(X_1,X_2)+bh^1(X_2,X_2)).
\end{align*}
\end{rem}

\begin{rem}\label{obs-simmetricnormalafin} Analogously,   {by} using the same notation   {as in} remark \ref{obs-normalafin}, we obtain   {from lemma \ref{lema-eq-equiafim} and  theorem \ref{teo-normalsymmafim}} the algorithm to compute   {a basis $\{\xi_1,\xi_2\}$ of} the symmetric equiaffine plane bundle:
\begin{align*}
  \xi_1 &= \nu_1- a X_1- b X_2, \\
  \xi_2 &= \nu_2- c X_1- d X_2,
\end{align*}
where:
\begin{align*}
  a &= \frac{2(a_2+a_3-3a_5-a_8) h^1(X_1,X_2)-(a_6+ a_7-a_1-3 a_4) h^1(X_2,X_2)}{4 h^1(X_1,X_2)^2+h^1(X_2,X_2)^2}, \\
  b &= \frac{2(a_6+ a_7-a_1-3 a_4) h^1(X_1,X_2)+(a_2+a_3-3a_5-a_8) h^1(X_2,X_2)}{4 h^1(X_1,X_2)^2+h^1(X_2,X_2)^2}, \\
  c &= -(a_1+a_4+b h^1(X_1,X_2)), \\
  d &= -(a_5+a_8+a h^1(X_1,X_2)+ b h^1(X_2,X_2)).
\end{align*}

\end{rem}

\section{Surfaces contained in a hypersurface}

  {We recall the definition of the Blaschke metric of an immersed hypersurface $N\subset \R^4$.} Let $\mathfrak{u}'=\{X_1,X_2,X_3\}$ be a tangent frame defined in some neighborhood $U$ of a point $p$ in $N$.  Now we consider
\[
H_{\mathfrak{\mathfrak{u}}'}(Y,Z)=[X_1,X_2,X_3,D_Z Y],\quad \forall Y,Z\in T_p N.
\]
  {Then $H_{\mathfrak{u}'}$ defines} a  symmetric bilinear form on $N$ that initially   {depends on} the tangent frame $\mathfrak{u}'$. However, if {we} suppose that $H_{\mathfrak{u}'}$ is non-degenerate, then we can normalize   {it} and the symmetric bilinear form
\[
\mathfrak G(Y,Z)=\frac{H_{\mathfrak{u}'}(Y,Z)}{(\det_{\mathfrak{u}'}H_{\mathfrak{u}'})^{\frac15}},\quad \forall Y,Z\in T_pN,
\]
does not  depend   {on the choice} of the tangent frame $\mathfrak{u}'$, where  $\det_{\mathfrak{u}'}H_{\mathfrak{u}'}=\det(H_{\mathfrak{u}'}(X_i,X_j)).$ The metric $\mathfrak G$ is called the  \textit{Blaschke metric} of $N$.
\medskip

  {If $N$ is locally strictly convex, then $H_{\mathfrak{u}'}$ is always non-degenerate and the tangent hyperplane $T_pN$ is a support hyperplane with a non-degenerate contact. In particular, given any immersed surface $M\subset N$ we have $T_pM \subset T_p N\subset \R^4$, and hence, $M$ is also locally strictly convex. Moreover, we can consider the Blaschke metric $\mathfrak G$ restricted to $M$.}

  {We claim that} we can choose a transversal vector field $\xi$ such that $g_{[\xi]}$ coincides with $\mathfrak G$ in $T_p M$. In fact, let $\mathfrak{u}=\{X_1,X_2\}$ be a frame in $T_p M$ and  we choose a tangent  vector field  $X_3\in T_p N$ such that $\mathfrak{u}'=\{X_1,X_2,X_3\}$ is a frame in $T_p N$, then
\[
G_{\mathfrak{u}}(Y,Z)=-H_{\mathfrak{u}'}(Y,Z),\quad \forall Y,Z\in T_pM.
\]
In particular, we have that $g_{[X_3]}=-\lambda \mathfrak G$ where $\lambda$ is given by
\[
\lambda=\frac{(\det_{\mathfrak{u}'}H_{\mathfrak{u}'})^{\frac15}}{(\det_\mathfrak{u}G_\mathfrak{u})^{\frac14}}.
\]
Then, it is enough to change the transversal vector field $X_3$ by $\xi=-X_3/\lambda^2$,   {so that} $g_{[\xi]}= \mathfrak G$.
\medskip

  {An interesting particular case} in this context are the immersed surfaces in affine hyperspheres.   {We recall here the definition of affine hypersphere (see for instance \cite{N} for details).}

\begin{defn} A hypersurface $H\subset \R^4$ is called an \emph{improper affine hypersphere} if the shape operator $S$ is identically 0. If $S=\lambda Id$, where $\lambda$ is a  nonzero constant, then $H$ is called a \emph{proper affine hypersphere}.
\end{defn}

\begin{exam}\label{example-paraboloide1}Immersed surface in an elliptic  paraboloid. We take  $M$   {as the} surface parameterized by
\[
X:(u,v)\mapsto (u,v,g(u,v),\frac12(u^2+v^2+g(u,v)^2)).
\]
Note that, $M$ is contained in an elliptic paraboloid $H$ given  by
\[H:(x,y,z)\mapsto(x,y,z,\frac12(x^2+y^2+z^2)).\]
From \cite{Li}, the Blaschke metric on $H$ is given by
\[\mathfrak{G}(e_i,e_j)=\delta_{ij},\]
where $e_1=(1,0,0,x)$, $e_2=(0,1,0,y)$ and $e_3=(0,0,1,z)$. Therefore, the Blaschke metric on $M$ is given by
\[\mathfrak{G}(X_u,X_u)=1+g_u^2,\quad\mathfrak{G}(X_u,X_v)=g_u g_v \hspace*{0.1cm}\mbox{ and }\hspace*{0.1cm}\mathfrak{G}(X_v,X_v)=1+g_v^2.\]
We can choose $\xi$ such that $g_{\xi}=\mathfrak{G}.$ By   {a simple computation,}
\[\xi=-\sqrt{1+g_u^2+g_v^2}(0,0,1,g).\]
\end{exam}

\begin{exam}\label{example-hiperboloid1}Immersed surface in   {a} hyperboloid of two sheets. We take  $M$   {as the} surface parameterized by
\[
X:(u,v)\mapsto (u,v,g(u,v),\sqrt{1+u^2+v^2+g(u,v)^2}).
\]
  {Then}, $M$ is contained in the hyperboloid of two sheets $H$
\[H:x_1^2+x_2^2+x_3^2-x_4^2=-1.\]
The Blaschke metric   {is} calculated in \cite[page 64]{Li}. We consider the metric field
  {\[\xi=\lambda(0,0,1,\frac{g}{\sqrt{1+u^2+v^2+g(u,v)^2}}),\]
where
\[\lambda=-\sqrt{1+g_u^2+g_v^2+(u g_u+v g_v-g(u,v))^2}.\]}
Then the Blaschke metric $\mathfrak{G}$ on $M$ coincides with $g_{\xi}$.    {It is not easy to check this computation by hand, but it is possible to do it with the aid of } the software Wolfram Mathematica. Explicitly the metric $g_{\xi}$ is given by
\begin{align*}
  g_{\xi}(X_u,X_u) &= (1+g_u^2) - \frac{(u+g g_u)^2}{1+u^2+v^2+g(u,v)^2}, \\
  g_{\xi}(X_u,X_v) &= g_u g_v - \frac{(u+g g_u)(v+ g g_v)}{1+u^2+v^2+g(u,v)^2}, \\
  g_{\xi}(X_v,X_v) &= (1+g_v^2) - \frac{(v+g g_v)^2}{1+u^2+v^2+g(u,v)^2}.
\end{align*}
\end{exam}

The examples \ref{example-paraboloide1} and \ref{example-hiperboloid1} are particular cases of affine hyperspheres  (hyperquadrics). The hypersurface  $Q(c,n)\subset \R^{n+1}$  (Li, Simon and Zhao, \cite{Li} ) is an example of an affine hypersphere which is not a hyperquadric:
\[Q(c,n): x_{n+1}=\frac{c}{x_1 x_2 \ldots x_n},\]
where $c=constant \neq 0$, $x_1>0, x_2>0, \ldots, x_n>0$.

\begin{exam}Immersed surface in $Q(1,3)$. We take  $M$   {as the} surface parameterized by
\[
X:(u,v)\mapsto (u,v,g(u,v),\frac{1}{u v g(u,v)}).
\]
Note that, $M$ is contained in $Q(1,3)$:
\[H:(x,y,z)\mapsto(x,y,z,\frac{1}{x y z}).\]
We consider the tangent frame  $\{e_1,e_2,e_3\}$ on $H$ to compute the Blaschke metric,  where $e_1 = (x,0,0,-\frac{1}{x y z)})$, $e_2 = (0,y,0,-\frac{1}{x y z})$, $e_3 = (0,0,z,-\frac{1}{x y z})$ and a transversal field $e_4=(x,y,z,\frac{-2}{x y z})$. By
  {taking derivatives,}
\begin{equation*}
    D_{e_1}e_1=(x,0,0,x,\frac{1}{x y z})=-e_1-2e_2-2e_3+2 e_4,
\end{equation*}
so $h_{11}=2$ which is the component $e_4$ of $D_{e_1}e_1$. Analogously we compute $h_{12}=h_{13}=h_{23}=1$ and $h_{22}=h_{33}=2$, therefore $\det(h_{ij})=4$. Finally
\[\mathfrak{G}_{ij}=\frac{h_{ij}}{(\det(h_{ij}))^{1/5}},\]
that is, $\mathfrak{G}_{ii}=2^{3/5}$ and $\mathfrak{G}_{ij}=\frac{1} {2^{2/5}}$ ($i\neq j$).   {We restrict} $\mathfrak{G}$ to the surface $M$ and obtain
\begin{equation*}
    \mathfrak{G}(X_u,X_u) = \frac{3 \times 2^{3/5}}{u^2}, \hspace*{0.5cm } \mathfrak{G}(X_u,X_v) = \frac{5}{2^{2/5}u v}, \hspace*{0.5cm } \mathfrak{G}(X_u,X_v)=\frac{3 \times 2^{3/5}}{v^2}.
\end{equation*}
It   {is enough} to consider the metric field
\[\xi=-\frac{\sqrt{2g^2+2(v g_v-u g_u)^2+(g+ v g_v+u g_u)^2}}{2^{4/5}g(u,v)}(0,0,g(u,v),\frac{-1}{u v g(u,v)}),\]
  {then we have} $g_{\xi}=\mathfrak{G}.$
\end{exam}

  {We have also the following obvious property, which will be used in the next section.}

\begin{prop}\label{prop-shape}Let $M\subset H \subset \R^4$ be an immersed surface in an affine hypersphere.   {Let $Y$ be} the affine normal to $H$ {and assume that $Y_p\in\sigma_p$, for all $p\in M$}. Then the shape operator $S_{Y}$ on $M$ is a multiple of the identity.
\end{prop}
\begin{proof}
  {Since $H$ is an affine hypersphere, by definition there is $\lambda\in\R$ such that $D_XY=-\lambda X$, for all $X\in T_pH$. In particular, this is also true for all $X\in T_pM$, hence $S_Y=\lambda Id$.}
%
\end{proof}

\section{Affine semiumbilical surfaces.}
Let $M \subset\R^4$    {be a} locally strictly convex   {surface with a transversal plane bundle $\sigma$. Given $\nu\in\sigma_p$, the $\nu$-principal curvatures at $p$ are the eigenvalues of the shape operator $S_\nu$ and the $\nu$-principal directions are the associated eigenvectors. It is common to call a point $p\in M$ umbilic if $S_{\nu}$ is a multiple of the identity, for all $\nu\in\sigma_p$. In analogy with the Euclidean case we introduce also the concept of semiumbilic point.}

\begin{defn} A point $p$ in $M \subset \R^4$ is called   {\emph{$\sigma$-semiumbilic} if $S_{\nu}$ is a multiple of the identity, for some non zero $\nu\in\sigma_p$.}
  {We say that $M$ is  \emph{$\sigma$-semiumbilical} if all its points are $\sigma$-semiumbilic.}

  {In the case that $\sigma$ is either the antisymmetric or the symmetric equiaffine plane bundle, then we say that $M$ is either \emph{antisymmetric} or \emph{symmetric} affine semiumbilical, respectively.}
\end{defn}

  {We see now that the semiumbilic points are related to the vanishing of the normal curvature tensor. We fix a metric field $\xi$ on $M$}.
 Let $\mathfrak{u}=\{X_1,X_2\}$   {be an orthonormal} tangent frame   {$\{\xi_1,\xi_2\}$ the corresponding transversal frame given by theorem \ref{vetores-unicos}}.  We consider $X=a X_1+ b X_2$ and $Y=c X_1+ d X_2$.  We write
\begin{align*}
  S_1 X_1 &= \lambda_1 X_1+ \lambda_2 X_2,  \\
  S_1 X_2 &= \lambda_3 X_1+ \lambda_4 X_2.
\end{align*}
We have
\[R_{\nabla^{\bot}}(X,Y)\xi_1=h(X,S_1 Y)-h(Y,S_1 X).\]
  {By using the }relations above and the bilinearity  of $h$, we prove that
\[R_{\nabla^{\bot}}(X,Y)\xi_1=(ad-bc)((\lambda_4-\lambda_1)h^1(X_1,X_2)-\lambda_2(h^1(X_2,X_2))\nu_1+(\lambda_3-\lambda_2)\nu_2).\]
Therefore, $R_{\nabla^{\bot}}(X,Y)\xi_1=0$ if and only if
\begin{align}
 (\lambda_4-\lambda_1)h^1(X_1,X_2) &= \lambda_2 h^1(X_2,X_2), \label{eq-shape1}\\
  \lambda_3 &= \lambda_2\label{eq-shape2}.
\end{align}
Analogously, if we write:
\begin{align*}
S_2 X_1 &= \mu_1 X_1+ \mu_2 X_2,\\
S_2 X_2 &= \mu_3 X_1+ \mu_4 X_2,
\end{align*}
then, $R_{\nabla^{\bot}}(X,Y)\xi_2=0$ if and only if
\begin{align}
 (\mu_4-\mu_1)h^1(X_1,X_2) &= \mu_2 h^1(X_2,X_2), \label{eq-shape21}\\
  \mu_3 &= \mu_2\label{eq-shape22}.
\end{align}

\begin{thm}\label{normal-curvature}
  {Let $p\in M$, then} $R_{\nabla^{\bot}}(p)\equiv 0$ if and only if the following conditions   {hold}:
\begin{itemize}
  \item the shape operator   {$S_{\nu}$ is self-adjoint $\forall \nu\in\sigma_p$,} and
  \item   {either $p$ is $\sigma$-semiumbilic and all the $\nu$-principal configurations agree with the asymptotic configuration or $p$ is an inflection}.
\end{itemize}
\end{thm}
\begin{proof}We suppose that $R_{\nabla^{\bot}}(p)\equiv 0$. By equations \eqref{eq-shape2} and \eqref{eq-shape22}   {it follows that} $\lambda_2=\lambda_3$ and $\mu_2=\mu_3$, in other words $S_1=S_{\nu_1}$ and  $S_2=S_{\nu_2}$ are self-adjoint. Now if $\nu=\alpha \nu_1+ \beta \nu_2$, then any $S_{\nu}=\alpha S_1+ \beta S_2$ is also self-adjoint.

If $p$  is not   {an} inflection, then $(b,c)=(h^1(X_1,X_2),h^1(X_2,X_2))\neq 0$, hence
\[\left|
    \begin{array}{cc}
      \lambda_4-\lambda_1 & \lambda_2 \\
      c & b \\
    \end{array}
  \right|=0 \Longleftrightarrow (\lambda_4-\lambda_1,\lambda_2)=t(c,b),
\]
\[ \left|
    \begin{array}{cc}
      \mu_4-\mu_1 & \mu_2 \\
      c & b \\
    \end{array}
  \right| =0\Longleftrightarrow (\mu_4-\mu_1,\mu_2)=s(c,b),\]
for some $t,s\in\R$, therefore
$\left|
    \begin{array}{cc}
      \lambda_4-\lambda_1 & \lambda_2 \\
      \mu_4-\mu_1 & \mu_2 \\
    \end{array}
  \right|=0.
$
\medskip\\
  {Moreover}, if $\nu=\alpha \nu_1+\beta \nu_2$, with $(\alpha,\beta)\neq 0$ then,
\[\alpha\left(
  \begin{array}{cc}
    \lambda_1 & \lambda_2 \\
    \lambda_3 & \lambda_4 \\
  \end{array}
\right)+ \beta \left(
                 \begin{array}{cc}
                   \mu_1 & \mu_2 \\
                   \mu_3 & \mu_4 \\
                 \end{array}
               \right)= \lambda \left(
                                  \begin{array}{cc}
                                    1 & 0 \\
                                    0 & 1 \\
                                  \end{array}
                                \right)\]
if and only if
\begin{align*}
  \alpha \lambda_1+ \beta \mu_1 &= \alpha \lambda_4+\beta \nu_4, \\
   \alpha \lambda_2+\beta \mu_2 &= 0,
\end{align*}
if and only if
\[\left|
    \begin{array}{cc}
      \lambda_4-\lambda_1 & \lambda_2, \\
      \mu_4-\mu_1 & \mu_2, \\
    \end{array}
  \right|=0.
\]
On the other   {hand}, the   {asymptotic configuration} is given by
\[\left|
    \begin{array}{ccc}
      y^2 & -xy & x^2 \\
      0 & b & c \\
      1 & 0 & 1 \\
    \end{array}
  \right|=-b x^2-c xy+by^2,
\]
and the $\nu$-principal configuration is given by:
$$\left|
    \begin{array}{ccc}
      y^2 & -xy & x^2 \\
      \alpha \lambda_1+\beta \mu_1 & \alpha \lambda_2+\beta \mu_2 & \alpha \lambda_4+\beta \mu_4 \\
      1 & 0 & 1 \\
    \end{array}
  \right|,
 $$
 which is equal to
$$-( \alpha \lambda_2+\beta \mu_2) x^2-(\alpha(\lambda_4-\lambda_1)+
  \beta(\mu_4-\mu_1)) xy+(\alpha \lambda_2+\beta \mu_2)y^2.
 $$
 Therefore, the   {asymptotic} configuration and the $\nu$-principal configuration are the same if and only if
\[(\alpha(\lambda_4-\lambda_1)+\beta(\mu_4-\mu_1))b- ( \alpha \lambda_2+\beta \mu_2)c=0,\]
or equivalently,
\begin{align*}
  b(\lambda_4-\lambda_1)-c\lambda_2 &= 0, \\
  b(\mu_4-\mu_1)-c\mu_2 &= 0.
\end{align*}
\end{proof}

In   {the last part of} this section we will consider   {an immersed surface $M$ in a locally strictly convex hyperquadric $N$}. These hyperquadrics are particular cases of affine hyperspheres (elliptic paraboloid, ellipsoid and hyperboloid of two sheets), \cite{Li}. By affine transformation, the locally strictly convex hyperquadrics are   {equivalent to one of the} following normal forms:
\begin{itemize}
  \item Elliptic paraboloid $x_4=\frac{1}{2}(x^2_1+x^2_2+x^2_3).$
  \item Ellipsoid $x^2_1+x^2_2+x^2_3+x^2_4=1.$
  \item Hyperboloid of two sheets $x^2_1+x^2_2+x^2_3-x^2_4=-1.$
\end{itemize}

\begin{thm}\label{thm-hiperquadric} Let $M \subset \R^4$ be a locally strictly convex surface  immersed in an hyperquadric $N$. Then the affine normal field to $N$ belongs to both the antisymmetric and symmetric equiaffine plane bundles of $M$, 	
  {with respect to} the Blaschke metric $\mathfrak{G}$ restricted to $M$.
\end{thm}

\begin{proof}
Let $p \in M$, since $M\subset  N$ there is a transversal vector field $\nu_1$ on $M$, which is tangent to $N$ and such that $g_{[-\nu_1]}$ coincides with the Blaschke metric $\mathfrak{G}$ restricted to surface $M$. Now we consider a local orthonormal tangent frame $\{X_1,X_2\}$ on $M$ relative to the metric $g_{[-\nu_1]}=\mathfrak{G}$. We fix a local transversal plane bundle $\sigma$, by theorem \ref{vetores-unicos} there is a local basis $\{\nu_1,\nu_2\}$ on $\sigma$ such that :
\[[X_1,X_2,\nu_1,\nu_2]=1,\hspace*{0.2cm}\hspace*{0.2cm} h^1(X_1,X_1)=0, \hspace*{0.2cm}\hspace*{0.2cm} h^2(X_i,X_j)=\delta_{ij}.\]
Now we consider the local frame, $\{e_1,e_2,e_3,e_4\}$ such that $e_1=X_1$, $e_2=X_2$, $e_3=\nu_1$ and $e_4= f Y$ where $Y$ is the affine normal vector field to $N$ and $f$ is defined by condition $[e_1,e_2,e_3,e_4]=1$. Since $\nu_2$ is a transversal vector field on $M$ and $[X_1,X_2,\nu_1,\nu_2]=1$, we can conclude $\nu_2=\lambda_3 e_3+ e_4,$ for some $\lambda_3$.

By theorem \ref{vetores-unicos} we can write:
\begin{align*}
  D_{e_{1}}e_1 &= a_1 e_1+ a_2 e_2+ \lambda_3 e_3+ e_4, \\
  D_{e_{1}}e_2 &= a_3 e_1+ a_4 e_2+ h^1(X_1,X_2)e_3, \\
  D_{e_{1}}e_3 &= \beta_1 e_1+ \beta_2 e_2+ (\tau_1^1(X_1)+\lambda_3 \tau_1^2(X_1))e_3+\tau_1^2(X_1)e_4, \\
  D_{e_{2}}e_1 &= a_5 e_1+a_6 e_2+h^1(X_1,X_2)e_3, \\
  D_{e_{2}}e_2 &= a_7 e_1+a_8 e_2+(h^1(X_2,X_2)+\lambda_3)e_3+e_4, \\
  D_{e_{2}}e_3 &= \beta_3 e_1+ \beta_4 e_2+(\tau_1^1(X_2)+\lambda_3\tau_1^2(X_2))e_3+\tau_1^2(X_2)e_4.
\end{align*}
We note that:
\begin{align*}
&h(e_1,e_1)=1,\ h(e_1,e_2)=0,\ h(e_1,e_3)=\tau_1^2(X_1),\\
&h(e_2,e_2)=1,\ h(e_2,e_3)=\tau_1^2(X_2).
\end{align*}
As $e_4$ is parallel to the affine normal, it follows:
\begin{align*}
  a_1+a_4+\tau_1^1(X_1)+\lambda_3 \tau_1^2(X_1) &= 0, \\
  a_5+a_8+\tau_1^1(X_2)+\lambda_3 \tau_1^2(X_2) &= 0,
\end{align*}
Since $N$ is a hyperquadric, then $C(X,Y,Z):=(\nabla_X h)(Y,Z)\equiv 0$ (see \cite{N}):
\begin{align*}
0 &=C(e_1,e_1,e_1)=e_1(h(e_1,e_1))-2h(\nabla_{e_1}e_1,e_1)\\
&=-2h(a_1 e_1+ a_2 e_2+ \lambda_3 e_3,e_1)\\
&=-2a_1-2 \lambda_3 \tau_1^2(X_1).\\
0&=C(e_1,e_1,e_2)=e_1(h(e_1,e_2))-h(\nabla_{e_1}e_1,e_2)-h(e_1,\nabla_{e_1}e_2)\\
&=-h(a_1 e_1+ a_2 e_2+ \lambda_3 e_3,e_2)-h(e_1,a_3 e_1+ a_4 e_2+ h^1(X_1,X_2)e_3)\\
&=-a_2-\lambda_3 \tau_1^2(X_2)-a_3-h^1(X_1,X_2)\tau_1^2(X_1).\\
0&=C(e_1,e_2,e_2)=e_1(h(e_2,e_2))-2h(\nabla_{e_1}e_2,e_2)\\
&=-2h(a_3 e_1+ a_4 e_2+ h^1(X_1,X_2)e_3 ,e_2)\\
&=-2 a_4 -2 h^1(X_1,X_2)\tau_1^2(X_2).
\end{align*}
\begin{align*}
0&=C(e_2,e_1,e_1)=e_2(h(e_1,e_1))-2h(\nabla_{e_2}e_1,e_1)\\
&=-2h(a_5 e_1+a_6 e_2+h^1(X_1,X_2)e_3,e_1)\\
&=-2a_5-2h^1(X_1,X_2)\tau_1^2(X_1).\\
0&=C(e_2,e_1,e_2)=e_2(h(e_1,e_2))-h(\nabla_{e_2}e_1,e_2)-h(e_1,\nabla_{e_2}e_2)\\
&=-h( a_5 e_1+a_6 e_2+h^1(X_1,X_2)e_3,e_2)-h(e_1,a_7 e_1+a_8 e_2+(h^1(X_2,X_2)+\lambda_3)e_3)\\
&=-a_6-h^1(X_1,X_2)\tau_1^2(X_2)-a_7-(h^1(X_2,X_2)+\lambda_3)\tau_1^2(X_1).\\
0 &=C(e_2,e_2,e_2)=e_2(h(e_2,e_2))-2 h(\nabla_{e_2}e_2,e_2)\\
&=-2h(a_7 e_1+a_8 e_2+(h^1(X_2,X_2)+\lambda_3)e_3,e_2)\\
&=-2 a_8-2(h^1(X_2,X_2)+\lambda_3)\tau_1^2(X_2).
\end{align*}
In the antisymmetric case we have:
\begin{align*}
a_8-a_2-a_3-a_5 =&-(h^1(X_2,X_2)+\lambda_3)\tau_1^2(X_2)+(\lambda_3 +h^1(X_1,X_2))\tau_1^2(X_1))\\
& +h^1(X_1,X_2)\tau_1^2(X_1)\\
=&-h^1(X_2,X_2)\tau_1^2(X_2)+2h^1(X_1,X_2)\tau_1^2(X_1),
\end{align*}
\begin{align*}
a_1-a_4-a_6-a_7 =&-\lambda_3\tau_1^2(X_1)+h^1(X_1,X_2)\tau_1^2(X_2)+h^1(X_1,X_2)\tau_1^2(X_2)\\
&+(h^1(X_2,X_2)+\lambda_3)\tau_1^2(X_1)\\
=&2h^1(X_1,X_2)\tau_1^2(X_2)+h^1(X_2,X_2)\tau_1^2(X_1).
\end{align*}
That is,
\begin{align*}
  2h^1(X_1,X_2)\tau_1^2(X_1)- h^1(X_2,X_2)\tau_1^2(X_2) &= a_8-a_2-a_3-a_5, \\
   h^1(X_2,X_2)\tau_1^2(X_1) + 2h^1(X_1,X_2)\tau_1^2(X_2)&= a_1-a_4-a_6-a_7.
\end{align*}
Therefore:
\begin{align*}
a_1+a_4+\tau_1^2(X_2)h^1(X_1,X_2)&=-\lambda_3\tau_1^2(X_1),\\
a_5+a_8+\tau_1^2(X_1)h^1(X_1,X_2)+\tau_1^2(X_2)h^1(X_2,X_2)&=-\lambda_3\tau_1^2(X_2).
\end{align*}

From remark \ref{obs-normalafin}, the affine normal plane is generated by the fields $\overline{\nu}_1,\overline{\nu}_2$, where:
\begin{align*}
  \overline{\nu}_1 &= \nu_1-\tau_1^2(X_1)X_1-\tau_1^2(X_2)X_2, \\
  \overline{\nu}_2 &= \nu_2-\lambda_3\tau_1^2(X_1) X_1- \lambda_3 \tau_1^2(X_2)X_2.
\end{align*}
By substituting $\overline{\nu}_2=\lambda_3 \nu_1+e_4-\lambda_3\tau_1^2(X_1) X_1- \lambda_3 \tau_1^2(X_2)X_2$, it follows that
\[\overline{\nu}_2=\lambda_3 \overline{\nu}_1+e_4.\]

Analogously, in the symmetric case:\\
$a_2+a_3-3a_5-a_8=a_2+a_3-2a_5- (a_5+a_8)$\\
$=-\lambda_3\tau_1^2(X_2)-h^1(X_1,X_2)\tau_1^2(X_1)+2h^1(X_1,X_2)\tau_1^2(X_1)+\tau_1^1(X_2)+\lambda_3 \tau_1^2(X_2)$\\
$=h^1(X_1,X_2)\tau_1^2(X_1)+\tau_1^1(X_2)$\\
$=h^1(X_1,X_2)\tau_1^2(X_1)+h^1(X_1,X_2)\tau_1^2(X_1)+h^1(X_2,X_2)\tau_1^2(X_2)$\\
$=2h^1(X_1,X_2)\tau_1^2(X_1)+h^1(X_2,X_2)\tau_1^2(X_2),$\\
\medskip\\
$a_6+a_7-3a_4-a_1=a_6+a_7-2 a_4-(a_1+a_4)$\\
$\hspace*{3.2cm}=h^1(X_1,X_2)\tau_1^2(X_2)-h^1(X_2,X_2)\tau_1^2(X_1)+\tau_1^1(X_1).$\\
$\hspace*{3.2cm}=h^1(X_1,X_2)\tau_1^2(X_2)-h^1(X_2,X_2)\tau_1^2(X_1)+h^1(X_1,X_2)\tau_1^2(X_2)$\\
$\hspace*{3.2cm}=2h^1(X_1,X_2)\tau_1^2(X_2)-h^1(X_2,X_2)\tau_1^2(X_1).$
\medskip\\
That is,
\begin{align*}
  2h^1(X_1,X_2)\tau_1^2(X_1)+h^1(X_2,X_2)\tau_1^2(X_2) &= a_2+a_3-3a_5-a_8, \\
  -h^1(X_2,X_2)\tau_1^2(X_1)+2h^1(X_1,X_2)\tau_1^2(X_2) &= a_6+a_7-3a_4-a_1.
\end{align*}
Therefore:\begin{align*}
&a_1+a_4+\tau_1^2(X_2)h^1(X_1,X_2)=-\lambda_3\tau_1^2(X_1),\\
&a_5+a_8+\tau_1^2(X_1)h^1(X_1,X_2)+\tau_1^2(X_2)h^1(X_2,X_2)=-\lambda_3\tau_1^2(X_2).
\end{align*}

From remark \ref{obs-simmetricnormalafin}, the affine normal plane is generated by the fields $\overline{\nu}_1,\overline{\nu}_2$ where:
\begin{align*}
  \overline{\nu}_1 &= \nu_1-\tau_1^2(X_1)X_1-\tau_1^2(X_2)X_2, \\
  \overline{\nu}_2 &= \nu_2-\lambda_3\tau_1^2(X_1) X_1- \lambda_3 \tau_1^2(X_2)X_2.
\end{align*}
By substituting $\overline{\nu}_2=\lambda_3 \nu_1+e_4-\lambda_3\tau_1^2(X_1) X_1- \lambda_3 \tau_1^2(X_2)X_2$, it follows
\[\overline{\nu}_2=\lambda_3 \overline{\nu}_1+e_4.\]
\end{proof}

\begin{rem}\label{rem:exam} When we consider the affine metric of Burstin and Mayer \cite{BM}, then theorem \ref{thm-hiperquadric} fails. In fact,
we suppose  that $M$ is parameterized by
\[
X(u,v)=(u,v,u v, \frac{u^2+ v^2+ u^2 v^2}{2})
\]
Note that $M$ is contained in the paraboloid $w=\frac12(x^2+y^2+z^2)$. Now
by using \textit{Wolfram Mathematica} we compute the affine normal plane of Nomizu and Vrancken \cite{N-V} which is generated by
\[\nu_1=\frac{1}{12(1+u^2)^{2/3}(1+v^2)^{2/3}}(u,v, 2 u v , 12+ 3 v^2+ u^2(13+ 14v^2)) \mbox{ and }\]
\[\nu_2=\frac{1}{12((1+u^2)(1+v^2))^{1/6}}(\frac{5v}{1+v^2},\frac{5u}{1+u^2},\frac{-12-7v^2-7u^2-2u^2v^2}{(1+u^2)(1+v^2)},-14 uv).\]
We can see that $(0,0,0,1)$ does not belong to the plane generated by $\nu_1$ and $\nu_2.$
\end{rem}

  {The following corollary is deduced from the proof of theorem \ref{thm-hiperquadric}.}

\begin{cor}   {With the same hypothesis as in theorem \ref{thm-hiperquadric}, the antisymmetric and symmetric equiaffine plane bundles of $M$ are equal.}
\end{cor}

Let $M\subset \R^4$ be an immersed surface in a hyperquadric $N$. We can consider on $M$ the   {extended Blaschke metric, by} writing $\mathfrak{G}(e_i,Y)=0$ for all $i=1,2,3.$ Here $\{e_1,e_2,e_3\}$ is an unimodular frame and $Y$ is the affine normal to $N$.

\begin{cor}   {With the same hypothesis as in} theorem \ref{thm-hiperquadric}, the antisymmetric (and symmetric) equiaffine plane to $M$ is the orthogonal plane to the tangent plane   {with respect to the extended} Blaschke metric $\mathfrak{G}.$
\end{cor}
\begin{proof}By theorem \ref{thm-hiperquadric}, the antisymmetric equiaffine plane is generated by
\begin{align*}
  \overline{\nu}_1 &= \nu_1-\tau_1^2(X_1)X_1-\tau_1^2(X_2)X_2, \\
  \overline{\nu}_2 &= \nu_2-\lambda_3\tau_1^2(X_1) X_1- \lambda_3 \tau_1^2(X_2)X_2.
\end{align*}
Now we consider $\mathfrak u'=\{X_1,X_2,\nu_1\}$ the tangent frame on $N$ and see that $\{\overline{\nu}_1,\overline{\nu}_2\}$ are orthogonal to the tangent plane:
\begin{align*}
\mathfrak{G}(\overline{\nu}_1,X_1)&=\mathfrak{G}(\nu_1,X_1)-\tau_1^2(X_1)\mathfrak{G}(X_1,X_1)\\
&=\displaystyle{\frac{H(\nu_1,X_1)}{(\det_{\mathfrak u'}H_{u'}) ^{1/5}}-\tau_1^2(X_1)\frac{H(X_1,X_1)}{(\det_{\mathfrak u'}H_{u'})^{1/5}}}\\
&=\displaystyle{\frac{\tau_1^2(X_1)}{(\det_{\mathfrak u'}H_{u'})^{1/5}}-\tau_1^2(X_1)\frac{1}{(\det_{\mathfrak u'}H_{u'})^{1/5}}=0}.
\end{align*}
Analogously,
\begin{align*}
&\mathfrak{G}(\overline{\nu}_1,X_2)=\mathfrak{G}(\nu_1,X_2)-\tau_1^2(X_2)\mathfrak{G}(X_2,X_2)=0,\\
&\mathfrak{G}(\overline{\nu}_2,X_1)=\mathfrak{G}(\overline{\nu}_2-e_4,X_1)=\lambda_3\mathfrak{G}(\overline{\nu}_1,X_1)=0,\\
&\mathfrak{G}(\overline{\nu}_2,X_2)=\mathfrak{G}(\overline{\nu}_2-e_4,X_2)=\lambda_3\mathfrak{G}(\overline{\nu}_1,X_2)=0.
\end{align*}
\end{proof}
\begin{cor}  {Any surface $M\subset \R^4$ immersed in a hyperquadric $N$  is antisymmetric (and symmetric) affine semiumbilical.}
\end{cor}
\begin{proof}
This result follows from theorem \ref{thm-hiperquadric} and proposition \ref{prop-shape}.
\end{proof}

\begin{exam} The product of plane curves is  also antisymmetric and symmetric affine semiumbilical with respect to some metric field. We consider the product of two plane curves parameterized by affine arc length, as
\[X(u,v)=(\alpha_1(u),\alpha_2(u),\beta_1(v),\beta_2(v))\]
and consider the transversal vector field $\xi=(0,\frac{1}{\alpha_1^{'}(u)},0,-\frac{1}{\beta_1^{'}(v)})$.\medskip\\
The metric of the transversal field $g_{\xi}$ is given by
\[g_{\xi}(X_u,X_u)=1, \hspace*{1.0cm}g_{\xi}(X_u,X_v)=0, \hspace*{1.0cm}g_{\xi}(X_v,X_v)=1.\]
Now consider the transversal plane bundle $\sigma=span\{X_{vv},X_{uu}\}$. By theorem \ref{vetores-unicos} we obtain
\begin{align*}
  \xi_1 &= (-\alpha_1^{''}(u),-\alpha_2^{''}(u),\beta_1^{''}(v),\beta_2^{''}(v)), \\
  \xi_2 &= (\alpha_1^{''}(u),\alpha_2^{''}(u),0,0).
\end{align*}
Since $\nabla_{X_i}X_j=0$ for $i=u,v$ and $j=u,v$ then $\nabla g=0$. Therefore, the plane generated by $\xi_1$ and $\xi_2$   {is the antisymmetric and the symmetric equiaffine plane}. Finally, by a simple calculation we note that the normal curvature tensor $R_{\nabla^{\bot}}\equiv 0$ and by theorem \ref{normal-curvature} it follows that the product of curves is  also antisymmetric (and symmetric) affine semiumbilical.
\end{exam}

\begin{exam}
In the case of immersed surfaces in affine hyperspheres $Q(c,n)$, in general the symmetric and antisymmetric equiaffine planes bundle are not  equal with respect to the Blaschke metric. Moreover, we have examples of immersed surfaces  in $Q(c,n)$ which are semiumbilical and some others which are not:
\begin{itemize}
  \item The surface parameterized by
\[(u,v)\mapsto (u,v,u v,\frac{1}{u^2 v^2})\]
is symmetric and antisymmetric affine semiumbilical,
  \item and the surface parameterized by
\[(u,v)\mapsto (u,v,v^2+ u^3,\frac{1}{u v (v^2+u^3)})\]
is not symmetric nor antisymmetric affine semiumbilical.
\end{itemize}
\end{exam}

\bibliographystyle{amsplain}

\end{document}